\documentclass[10pt]{amsart}

\usepackage{amsmath, amsfonts, amsthm, amssymb, graphicx, fullpage, enumerate}

\newtheorem{theorem}{Theorem}[section]      
\newtheorem{lemma}[theorem]{Lemma}
\newtheorem{corollary}[theorem]{Corollary}

\newtheorem{proposition}[theorem]{Proposition}
  
\newtheorem*{main theorem}{Main Theorem}     
\newtheorem*{thmA}{Theorem A}

\theoremstyle{remark}     
   
\theoremstyle{definition}  
\newtheorem*{definition}{Definition}   
      
\def\T{\mathbb{T}}         
\def\N{\mathbb{N}}

\def\R{\mathbb{R}}     
\def\Z{\mathbb{Z}}
     
\def\C{\mathbb{C}}  
\def\P{\mathcal{P}}

\def\bar#1{\overline{#1}} 

\def\norm#1{\|#1\|}

\begin{document}
\title{Binary Quadratic Forms in Difference Sets} 
\author{Alex Rice}

\address{Department of Mathematics, Millsaps College, Jackson, MS 39210}
\email{riceaj@millsaps.edu} 
\subjclass[2000]{11B30}
\begin{abstract} We show that if $h(x,y)=ax^2+bxy+cy^2\in \Z[x,y]$ satisfies $\Delta(h)=b^2-4ac\neq 0,$ then any subset of $\{1,2,\dots,N\}$ lacking nonzero differences in the image of $h$ has size at most a constant depending on $h$ times $N\exp(-c\sqrt{\log N})$, where $c=c(h)>0$. We achieve this goal by adapting an $L^2$ density increment strategy previously used to establish analogous results for sums of one or more single-variable polynomials. Our exposition is thorough and self-contained, in order to serve as an accessible gateway for readers who are unfamiliar with previous implementations of these techniques. 
\end{abstract}
\maketitle 
\setlength{\parskip}{5pt}   

\section{Introduction} 

Established independently by S\'ark\"ozy and Furstenberg during the 1970s, settling a question of Lov\'asz, it is a well-studied fact that any set of integers of positive upper density necessarily contains two distinct elements that differ by a perfect square. Equivalently, if $A\subseteq \N$ contains no such pair, then $$\lim_{N\to \infty} \frac{|A\cap[1,N]|}{N}=0.$$ Here we use $[1,N]$ to denote $\{1,2,\dots, N\}$ and $|X|$ to denote the size of a finite set $X$. Furstenberg \cite{Furst} achieved this result qualitatively via ergodic theory, specifically his correspondence principle, but obtained no information on the rate at which the density must decay, while S\'ark\"ozy \cite{Sark1} employed a Fourier analytic density increment strategy to show that if $A\subseteq [1,N]$ has no square differences, then \begin{equation}\label{sark1bound} \frac{|A|}{N} \ll \left(\frac{(\log \log N)^2}{\log N}\right)^{1/3}.\end{equation} Throughout the paper we use $\log$ to denote the natural logarithm, and we use ``$\ll$" to denote ``less than a constant times", with subscripts indicating on what parameters, if any, the implied constant depends. S\'ark\"ozy's argument was driven by the Hardy-Littlewood circle method, and was inspired by Roth's \cite{Roth} proof that sets of positive upper density contain three-term arithmetic progressions.

Using a more intricate Fourier analytic argument, Pintz, Steiger, and Szemer\'edi \cite{PSS} improved (\ref{sark1bound}) to \begin{equation}\label{PSSbound} |A| \ll N(\log N)^{-c\log\log\log\log N}, \end{equation} with $c=1/12$. While more elementary Fourier analytic proofs \cite{Green,Lyall} and a Fourier-free density increment proof \cite{taoblog} have also been discovered, it is versions of these two Fourier analytic attacks that have yielded the best quantitative information. In the ensuing decades, these two methods have been refined and applied to other sets of prohibited differences, such as more general polynomial images \cite{BPPS,Slip,Lucier,HLR}, shifted primes \cite{Sark3,Lucier2,Ruz}, polynomial curves in higher-dimensional integer lattices \cite{LM}, and images of the primes under polynomials \cite{lipan,Rice2}. 

With regard to sums of one or more single-variable polynomials, the author \cite{Rice} pushed these two methods to their breaking points. In the case of one single-variable polynomial, if $h\in \Z[x]$ has degree $k\geq 2$ and $h(\N)$ contains a multiple of $q$ for every $q\in \N$, known as an \textit{intersective polynomial}, then any set $A\subseteq [1,N]$ with no nonzero differences in the image of $h$ satisfies (\ref{PSSbound})  for any $c<(\log((k^2+k)/2))^{-1}$, with the implied constant depending on $h$ and $c$. The intersective condition is necessary to force any density decay, as otherwise one can take $A=q\N$ if $h(\N)$ misses $q\Z$, and in that sense this is a maximal extension of the elaborate techniques developed in \cite{PSS} and \cite{BPPS}. 

Further, if we allow the additional degree of freedom of a second polynomial, then the more straightforward density increment approach yields density bounds that are even better than (\ref{PSSbound}), as described below.

\begin{thmA}[\cite{Rice}]Suppose $g, h \in \Z[x]$ are nonzero intersective polynomials and $A\subseteq [1,N]$. If $$a-a'\neq g(m)+h(n)$$ for all distinct pairs $a,a'\in A$ and all $m,n \in \N$, then \begin{equation*}|A| \ll_{g,h} Ne^{-c(\log N)^{\mu}},  \end{equation*} where $c=c(g,h)>0$, $\mu=\mu(\deg(g),\deg(h))>0$, and $\mu(2,2)=1/2$.
\end{thmA}

\noindent As a notable example, Theorem A establishes an upper bound of $\exp(-c\sqrt{\log N})$ for the density of subsets of $[1,N]$ lacking differences that are the sum of two squares. There is also a brief discussion of sums of three or more single-variable polynomials at the end of \cite{Rice}, but the improvements in density bounds are modest as $\exp(-c\sqrt{\log N})$ arises as a natural limit of the method, a fact that we discuss in Section \ref{boundsec}. 

While the generality of Theorem A is pleasing,  prohibited differences of the form $g(m)+h(n)$ can be thought of as the diagonal special case of differences of the form $h(m,n)$ where $h\in \Z[x,y]$. Of course, if $h(x,y)=\tilde{h}(g(x,y))$ for some $g\in \Z[x,y]$ and $\tilde{h}\in \Z[x]$ with $\deg(\tilde{h})\geq 2$, then the image of $h$ is contained in the image of $\tilde{h}$, in which case we could not hope to improve on the original setting of one single-variable polynomial. However, in other cases, we expect that the freedom of two variables should allow for improved density bounds. It is with this expectation in mind that we gently wade into the arena of potentially non-diagonal two-variable polynomials by exploring the following natural generalization of the aforementioned special case $m^2+n^2$.  

\begin{definition} $h\in \Z[x,y]$ is called a \textit{binary quadratic form} if $$h(x,y)=ax^2+bxy+cy^2$$ for some $a,b,c \in \Z$. Further, we define the \textit{discriminant} of $h$ by $$\Delta(h)=b^2-4ac,$$ noting that $h(x,y)=d(rx+sy)^2$ for some $d,r,s \in \Z$ if and only if $\Delta(h)=0$.

\end{definition}

Our main result is the following, which says that under the necessary restriction that a binary quadratic form does not collapse into a dilated perfect square, we achieve the same density bounds previously established in the diagonal case, which are likely the best possible for our chosen method.

\begin{theorem} \label{main} Suppose $h \in \Z[x,y]$ is a binary quadratic form with $\Delta(h)\neq 0$. If $A\subseteq [1,N]$ with $$a-a'\neq h(m,n)$$ for all distinct pairs $a,a'\in A$ and all $m,n \in \N$, then $$|A| \ll_{h} Ne^{-c\sqrt{\log N}}, $$ where $c=c(h)>0$.
\end{theorem}

We note that the image of every nonzero binary quadratic form contains a dilation of the squares, and hence our result is only material because the established density bound is better than (\ref{PSSbound}). Our goal for the remainder of the paper is twofold: to establish Theorem \ref{main}, which we hope will serve as a starting point for the application of these methods to more general polynomials in several variables, and to provide thorough and self-contained exposition of all of the components of this iteration scheme for those unfamiliar with its previous applications, such as the original case of the squares.   

\section{Main Iteration Lemma: Deducing Theorem \ref{main}} \label{itsec} The principle behind a density increment strategy is that a set which lacks the desired  arithmetic structure should spawn a new, significantly denser subset of a slightly smaller interval with an inherited lack of arithmetic structure. Iterating this procedure enough times for the density to reach $1$ provides an upper bound on the density of the original set.

For this section, we fix a binary a quadratic form $h\in \Z[x,y]$  with $\Delta(h)\neq 0$, and we let $$I(h)=\{ h(m,n): m,n\in \N\} \setminus \{0\}. $$ Our iteration scheme is encapsulated by the following lemma, from which we quickly deduce Theorem \ref{main}.

\begin{lemma}\label{mainit} Suppose $A\subseteq [1,N]$ with $|A|=\delta N$ and $\delta \geq N^{-1/20}.$ If $(A-A)\cap I(h)= \emptyset$, then there exists $A'\subseteq [1,N']$ with $|A'|=\delta'N'$ and a constant $c=c(h)>0$ with 
\begin{equation*} N'\gg_h \delta^4 N, \quad  \delta' \geq (1+c)\delta, \quad \text{and} \quad (A'-A')\cap I(h)=\emptyset. \\\end{equation*}

\end{lemma}

\subsection{Proof of Theorem \ref{main}} \label{mainproof} Suppose $A \subseteq [1,N]$ with $|A|=\delta N$ and $(A-A)\cap I(h)= \emptyset$. Setting $A_0=A$, $N_0=N$, and $\delta_0=\delta$, Lemma \ref{mainit} yields, for each $m$, a set $A_m \subseteq [1,N_m]$ with $|A_m|=\delta_mN_m$ and $(A_m-A_m)\cap I(h)= \emptyset$ satisfying
\begin{equation} \label{Nm} N_m \geq c\delta^4N_{m-1} \geq (c\delta^4)^m N
\end{equation}
and
\begin{equation} \label{incsize} \delta_m \geq (1+c)\delta_{m-1} \geq (1+c)^m\delta
\end{equation}
as long as 
\begin{equation} \label{delm} \delta_m \geq N_m^{-1/20}.
\end{equation}
By (\ref{incsize}), we see that the density $\delta_m$ will surpass $1$, and hence (\ref{delm}) must fail, for $m=C\log(\delta^{-1})$. In particular, by (\ref{Nm}) we have 
\begin{equation*} \delta \leq (c\delta^4)^{-C\log(\delta^{-1})}N^{-1/20}, 
\end{equation*}
which can be rearranged to 
\begin{equation*} N\leq e^{C\log^2(\delta^{-1})}
\end{equation*} 
and hence implies
\begin{equation*} \delta \ll_h e^{-c\sqrt{\log N}},
\end{equation*}
as required. \qed

\subsection{Roadmap for the remainder of the paper} \label{roadmap} Our task is is now completely reduced to a proof of Lemma \ref{mainit}, the two major components of which are described below. \begin{itemize}\item[(i)] The condition $(A-A)\cap I(h)=\emptyset$ represents unexpected, nonuniform behavior, which we  expect to be detectable in the Fourier analytic behavior of $A$. More specifically, we use orthogonality of characters and adaptations of standard exponential sum estimates to locate a single small denominator $q$ such that the Fourier transform of the characteristic function of $A$, translated to have mean value zero, has substantial $L^2$ concentration near rationals with denominator $q$. The Fourier analytic infrastructure is introduced in Section \ref{faz}, the proof of this component is carried out in Section \ref{sqL2proof}, and the required exponential sum estimates are exposed in great detail in Section \ref{expest}.  \\ \item[(ii)] The substantial $L^2$ concentration of the transform of the translated characteristic function of $A$ near rationals with a particular denominator $q$ indicates a correlation of $A$ with a linear phase function. In particular, we show that this implies that $A$ has significantly increased relative density on a long arithmetic progression $P$ of step size $q$. Since this implication has nothing to do with $h$, or any other assumptions about $A$, we prove a general version preemptively in Section \ref{dincsec}. Finally, by shifting and rescaling the intersection of $A$ with a subprogression of $P$ of step size $q^2$, we obtain our new, denser set $A'$ with $(A'-A')\cap I(h)=\emptyset$. The complete deduction of Lemma \ref{mainit} from these two components is carried out in Section \ref{mainitproof}.\end{itemize}    

\subsection{A discussion of novelty and bounds} \label{boundsec} As indicated in the introduction, the procedure outlined in Section \ref{roadmap}, though refined over the years, goes back to S\'ark\"ozy in the 1970s. The improvement in bounds in Theorem \ref{main} and Theorem A, as compared to the case of one single-variable polynomial, comes from the details of the numerology in Lemma \ref{mainit}, most notably the size of the density increment $\delta' \geq (1+c)\delta$. This effectively optimal increase in density is facilitated by the quality of the exponential sum estimates mentioned in item (i) above. 

More specifically, the size of the density increment can be traced to the level of decay achieved in normalized complete local exponential sums. In the original setting of square differences, for example, the relevant decay comes from the standard estimate \begin{equation} \label{novel1} \left|\frac{1}{q}\sum_{r=0}^{q-1} e^{2\pi i r^2 a/q}\right| \ll q^{-1/2} \end{equation} for $(a,q)=1$, which ultimately leads to a density increment $\delta' \geq \delta+c\delta^2$. Substituting this increment size, and other minor necessary modifications, into the proof in Section \ref{mainproof} leads to the upper bound $$\delta \ll \frac{\log\log N}{\log N}, $$ which is better than S\'ark\"ozy's original result (\ref{sark1bound}). The reader may refer to \cite{LR} or \cite{thesis} for  full expositions of this refinement in the cases of squares, shifted primes, and, in the latter case, intersective polynomials.

In the case of sums of two squares, covered in Theorem A, the relevant decay comes from the analogous two-variable sum that then splits, allowing one to use the same estimate (\ref{novel1}) to conclude \begin{equation*}\left|\frac{1}{q^2}\sum_{r,s=0}^{q-1} e^{2\pi i (r^2+s^2) a/q}\right| = \left|\frac{1}{q}\sum_{r=0}^{q-1} e^{2\pi i r^2 a/q}\right|^2 \ll  q^{-1} \end{equation*} for $(a,q)=1$, which is good enough to get the optimal density increment. The novelty of Theorem \ref{main} is rooted in the fact that when $\Delta(h)\neq 0$, we get the same level of decay, namely \begin{equation*} \left|\frac{1}{q^2}\sum_{r,s=0}^{q-1} e^{2\pi i h(r,s) a/q}\right| \ll_h q^{-1} \end{equation*} for $(a,q)=1$, even though the sum no longer necessarily splits.

In order to improve on the bound $\exp(-c\sqrt{\log N})$ using this approach, for any fixed set of prohibited differences, one of two components of the numerology of Lemma \ref{mainit} must be improved: either the ratio $N'/N$ must decay more slowly than any power of $\delta$, or the ratio $\delta'/\delta$ must tend to infinity, as $\delta\to 0$, neither of which appear feasible in any nontrivial context. However, the question of whether the known upper bounds are even remotely sharp remains completely open in all of the aforementioned cases. For a more detailed discussion of lower bounds, constructions, and conjectures, the reader may refer to Section 1.4 of \cite{Rice}. 

\section{Preliminaries} 

\subsection{Fourier analysis and the circle method on $\Z$} \label{faz} We embed our finite sets in $\Z$, on which we utilize the discrete Fourier transform. Specifically, for a function $F: \Z \to \C$ with finite support, we define $\widehat{F}: \T \to \C$, where $\T$ denotes the  circle parametrized by the interval $[0,1]$ with $0$ and $1$ identified, by \begin{equation*} \widehat{F}(\alpha) = \sum_{n \in \Z} F(n)e^{-2 \pi in\alpha}. \end{equation*} In this finite support context, Plancherel's Identity \begin{equation} \label{planch} \sum_{n\in \Z} |F(n)|^2=\int_0^1 |\widehat{F}(\alpha)|^2 d \alpha \end{equation} is a direct consequence of the orthogonality relation \begin{equation} \label{orthrel} \int_0^1 e^{2\pi i n \alpha} d\alpha = \begin{cases} 1 &\text{if } n=0 \\ 0 &\text{if } n \in \Z \setminus \{0\}. \end{cases} 
\end{equation}

\noindent Given $N\in \N$ and a set $A\subseteq [1,N]$ with $|A|=\delta N$, we examine the Fourier analytic behavior of $A$ by considering the \textit{balanced function}, $f_A$, defined by
\begin{equation*} f_A=1_A-\delta 1_{[1,N]}.\end{equation*} 
We analyze $\widehat{f_A}$, and other arising exponential sums, using the Hardy-Littlewood circle method, decomposing the frequency space into two components: the set of points on the circle that are close to rationals with small denominator, and the complement. 

\begin{definition}Given $N\in\N$ and $\eta>0$, we define, for each $q\in \N$ and $a\in [1,q]$,
\begin{equation*} \mathbf{M}_{a/q}=\mathbf{M}_{a/q}(N,\eta)=\left\{ \alpha \in \T : |\alpha-\frac{a}{q}| < \frac{1}{\eta^2 N} \right\}, \ \ \mathbf{M}_q=\bigcup_{(a,q)=1} \mathbf{M}_{a/q}, \ \text{ and } \ \mathbf{M}'_q = \bigcup_{r\mid q} \mathbf{M}_q = \bigcup_{a=1}^q \mathbf{M}_{a/q}.
\end{equation*}
We then define $\mathfrak{M}$, the \textit{major arcs} and $\mathfrak{m}$, the \textit{minor arcs}, by
\begin{equation*} \mathfrak{M}=\bigcup_{q=1}^{\eta^{-1}} \mathbf{M}_q, \quad \mathfrak{m}=\T\setminus \mathfrak{M}.
\end{equation*} We note that if $\eta^2 N > 2Q^2$, then \begin{equation} \label{majdisj}\mathbf{M}_{a/q}\cap\mathbf{M}_{b/r}=\emptyset \end{equation}whenever  $a/q\neq b/r$ and  $q,r \leq Q$. 
\end{definition}

\subsection{Density increment lemma} \label{dincsec} The following standard result shows that for  $A\subseteq [1,N]$, $L^2$ concentration of $\widehat{f_A}$ near rationals with a particular denominator $q$ implies increased relative density on a long arithmetic progression of step size $q$, as described in item (ii) in Section \ref{roadmap}.

\begin{lemma} \label{dinc} Suppose $A \subseteq [1,N]$ with $|A|=\delta N$. If  $q \in \N$, $\sigma,\eta>0$, and
\begin{equation*} \int_{\mathbf{M}'_q}|\widehat{f_A}(\alpha)|^2d\alpha \geq \sigma\delta^2 N,
\end{equation*} 
then there exists an arithmetic progression 
\begin{equation*}P=\{x+\ell q : 1\leq \ell \leq L\}
\end{equation*}
with $qL \gg \min\{\sigma, \eta^2\} N $ and $|A\cap P| \geq (1+\sigma/32)\delta L$.
\end{lemma}

\begin{proof}
Suppose $A\subseteq [1,N]$ with $|A|=\delta N$, $\sigma, \eta>0$. Suppose further that  
\begin{equation}\label{qmass} \int_{\mathbf{M'}_q}|\widehat{f_A}(\alpha)|^2d\alpha \geq \sigma\delta^2 N,
\end{equation}  
and let $P = \{q,2q, \dots, Lq\}$ with $L= \lfloor \min\{\sigma, \eta^2\} N/128 q \rfloor$. We will show that some translate of $P$ satisfies the conclusion of Lemma \ref{dinc}. We note that for $\alpha \in [0,1]$,
\begin{equation}\label{Phat} |\widehat{1_{P}}(\alpha)|=\Big| \sum_{\ell=1}^L e^{-2\pi i \ell q \alpha} \Big| \geq L-\sum_{\ell=1}^L |1-e^{-2\pi i \ell q \alpha}|\geq L-2\pi L^2 \norm{q\alpha},
\end{equation} where $\norm{\cdot}$ denotes the distance to the nearest integer. Further, if $\alpha \in \mathbf{M}'_q$, then \begin{equation}\label{dni} \norm{q\alpha} \leq \frac{q}{\eta^2 N} \leq \frac{1}{4\pi L}.
\end{equation} Therefore, by (\ref{Phat}) and (\ref{dni}) we have 
\begin{equation}\label{L/2} |\widehat{1_{P}}(\alpha)| \geq L/2 \quad \text{for all} \quad \alpha \in \mathbf{M}'_q.
\end{equation}
By (\ref{qmass}), (\ref{L/2}), and Plancherel's Identity (\ref{planch}) we see
\begin{equation}\label{convmass}  \sigma\delta^2 N \leq \int_{\mathbf{M}'_q}|\widehat{f_A}(\alpha)|^2d\alpha \leq \frac{4}{L^2} \int_0^1|\widehat{f_A}(\alpha)|^2|\widehat{1_{P}}(\alpha)|^2d\alpha=\frac{4}{L^2}\sum_{n\in \Z} |f_A * \widetilde{1_{P}}(n)|^2,
\end{equation} where  $\widetilde{1_{P}}(n)= 1_{P}(-n)$ and  
\begin{equation}\label{conv} f_A * \widetilde{1_{P}}(n)=\sum_{m\in \Z}f_A(m)1_{P}(m-n)= |A\cap(P+n)|-\delta|(P+n)\cap[1,N]|.
\end{equation}
We now take advantage of the fact that $f_A$, and consequently $f_A * \widetilde{1_{P}}$, has mean value zero. In other words,
\begin{equation} \label{mv0} \sum_{n\in\Z} f_A * \widetilde{1_{P}}(n)=0.
\end{equation} As with any real valued function, we can write \begin{equation} \label{pospart} |f_A * \widetilde{1_{P}}|= 2(f_A * \widetilde{1_{P}})_+ -f_A * \widetilde{1_{P}},\end{equation} where $(f_A * \widetilde{1_{P}})_+=\max\{f_A * \widetilde{1_{P}},0\}$. 

\noindent For the purposes of proving Lemma \ref{dinc}, we can assume that $f_A * \widetilde{1_{P}}(n) \leq 2\delta L$ for all $n\in \Z$, as otherwise the progression $P+n$ would more than satisfy the conclusion. Combined with the trivial upper bound $f_A * \widetilde{1_{P}}(n)\geq -\delta L$, we can assume 
\begin{equation} \label{linfty} |f_A * \widetilde{1_{P}}(n)| \leq 2\delta L \quad \text{for all} \quad n\in \Z.
\end{equation}
By (\ref{convmass}), (\ref{mv0}), (\ref{pospart}), and (\ref{linfty}), we have
\begin{equation}\label{posmass} \sum_{n\in\Z} (f_A * \widetilde{1_{P}})_+(n) = \frac{1}{2} \sum_{n\in \Z} |f_A * \widetilde{1_{P}}| \geq \frac{1}{4\delta L} \sum_{n\in \Z} |f_A * \widetilde{1_{P}}|^2\geq\frac{\sigma\delta NL}{16}. 
\end{equation}
By (\ref{conv}), we see that $f_A * \widetilde{1_{P}}(n)=0$ if $n\notin [-qL,N]$. Letting $E= \{ n\in \Z : P+n \subseteq [1,N] \}$ and $F=[-qL,N]\setminus E$, we see that $|F| \leq 2qL$. Therefore, by (\ref{linfty}), (\ref{posmass}), and the bound $128qL\leq \sigma N$, we have 
\begin{equation} \label{EF} \sum_{n\in E} (f_A * \widetilde{1_{P}})_+(n)  \geq \frac{\sigma\delta NL}{16}-2\delta L|F| \geq \frac{\sigma\delta NL}{16}-4q\delta L^2> \frac{\sigma\delta NL}{32}.
\end{equation}  Recalling that $|E|\leq N$ and $f_A * \widetilde{1_{P}}(n)= |A\cap (P+n)| - \delta L$ for all $n\in E$, we have that there exists $n\in \Z$ with $$|A\cap(P+n)| \geq (1 + \sigma/32)\delta L,$$ as required. \end{proof}

\section{$L^2$ Concentration} For this section, we once again fix a binary a quadratic form $h\in \Z[x,y]$ with $\Delta(h)\neq 0$, and  let $$I(h)=\{ h(m,n): m,n\in \N\} \setminus \{0\}. $$ The following result makes precise the implication outlined in item (i) in Section \ref{roadmap}, in which the condition $(A-A)\cap I(h)=\emptyset$ forces substantial $L^2$ concentration of $\widehat{f_A}$ near rationals with a single small denominator. Combining this with Lemma \ref{dinc}, we then quickly deduce Lemma \ref{mainit}.

\begin{lemma}  \label{sqL2} Suppose $A\subseteq [1,N]$ with $|A|=\delta N$, and let $\eta=c_0\delta$ for a sufficiently small constant $c_0=c_0(h)>0$. If $(A-A)\cap I(h)=\emptyset$, $\delta \geq N^{-1/20}$, and $|A\cap(N/9,8N/9)|\geq 3\delta N/4$, then there exists $q\leq \eta^{-1}$ such that 
\begin{equation*} \int_{\mathbf{M}'_q} |\widehat{f_A}(\alpha)|^2d\alpha \gg_h \delta^2 N.
\end{equation*}
\end{lemma}

\subsection{Proof of Lemma \ref{mainit}} \label{mainitproof} Suppose $A\subseteq [1,N]$, $|A|=\delta N$, $\delta \geq N^{-1/20}$, and $(A-A)\cap I(h)=\emptyset$.

\noindent If $|A\cap (N/9,8N/9)| < 3\delta N/4$, then $\max \{ |A\cap[1,N/9]|, |A\cap [8N/9,N]| \} > \delta N/8$. In other words, $A$ has density at least $9\delta/8$ on one of these intervals.

\noindent Otherwise, Lemmas \ref{sqL2} and \ref{dinc} apply, so in either case, letting $\eta=c_0\delta$, there exists $q\leq \eta^{-1}$ and an arithmetic progression 
\begin{equation*}P=\{x+\ell q : 1\leq \ell \leq L\}
\end{equation*}
with $qL\gg_h \delta^2 N$ and $|A\cap P| \geq (1+c)\delta L$. Partitioning $P$ into subprogressions of step size $q^2$, the pigeonhole principle yields a progression 
\begin{equation*} P'=\{y+\ell q^2 : 1\leq \ell \leq N'\} \subseteq P
\end{equation*}
with $N'\geq L/2q$ and $|A\cap P'| \geq (1+c)\delta N'$. This allows us to define a set $A' \subseteq [1,N']$ by \begin{equation*} A' = \{\ell \in [1,N'] : y+\ell q^2 \in A \},
\end{equation*} which clearly satisfies $|A'| \geq (1+c)\delta N'$ and $N'\gg_h \delta^2 N/q^2 \gg_h \delta^4N$. Moreover, since $q^2h(m,n)=h(qm,qn)$, $A'$ inherits the lack of $h(m,n)$ differences from $A$. \qed

\

\noindent Our task is now completely reduced to a proof of Lemma \ref{sqL2}.

\subsection{Proof of Lemma \ref{sqL2}} \label{sqL2proof} Suppose $A\subseteq [1,N]$ with $|A|=\delta N$, and let $\eta=c_0\delta$. 

\noindent We let $J=|b_1|+|b_2|+|b_3|$, $M=\sqrt{N/9J}$, $Z=\{(m,n) \in [1,M]^2: h(m,n)=0\}$, and  $\Lambda=[1,M]^2\setminus Z$. 

\noindent We note that \begin{equation}\label{z} |Z| \ll_h M. \end{equation}

\noindent If $(A-A)\cap I(h)=\emptyset$, then since $h(\Lambda)\subseteq [-N/9,N/9]$, we see that 
\begin{align*} \sum_{\substack{x\in \Z \\ (m,n)\in \Lambda}} f_A(x)f_A(x+h(m,n)) =  &\sum_{\substack{x\in \Z \\ (m,n)\in \Lambda}} 1_A(x)1_A(x+h(m,n)) - \delta  \sum_{\substack{x\in \Z \\ (m,n)\in \Lambda}} 1_A(x)1_{[1,N]}(x+h(m,n)) \\ -\delta  &\sum_{\substack{x\in \Z \\ (m,n)\in \Lambda}}  1_{[1,N]}(x-h(m,n))1_A(x) + \delta^2  \sum_{\substack{x\in \Z \\ (m,n)\in \Lambda}} 1_{[1,N]}(x)1_{[1,N]}(x+h(m,n)) \\ \leq &\Big(\delta^2 N - 2\delta |A\cap(N/9,8N/9)| \Big)|\Lambda|.
\end{align*}
Therefore, if $|A\cap(N/9,8N/9)| \geq 3\delta N/4$, we have
\begin{equation}\label{neg} \sum_{\substack{n\in \Z \\ 1\leq m \leq M}} f_A(n)f_A(x+h(m,n)) \leq -\delta^2N|\Lambda|/2.
\end{equation}
One can check using orthogonality (\ref{orthrel}) and Plancherel's Identity (\ref{planch}) that 
\begin{align*}\label{trans} \sum_{\substack{x\in \Z \\ (m,n)\in \Lambda}} f_A(x)f_A(x+h(m,n)) &= \sum_{\substack{x,y\in \Z \\ (m,n)\in \Lambda}} f_A(x)f_A(y)\int_{0}^1e^{2\pi i (x-y+h(m,n))\alpha} d\alpha \\ &= \int_0^1 \left(\sum_{x\in \Z}f_A(x)e^{2\pi i x \alpha} \right) \left(\sum_{y \in \Z} f_A(y) e^{-2\pi i y \alpha} \right) \left(\sum_{(m,n)\in \Lambda} e^{2\pi i h(m,n)\alpha}\right) d\alpha \\ &=\int_0^1 |\widehat{f_A}(\alpha)|^2S_M(\alpha) d\alpha +O(\delta N |Z|), 
\end{align*}
where 
\begin{equation*} S_x(\alpha)=\sum_{1\leq m,n \leq x} e^{2\pi i h(m,n)\alpha}.
\end{equation*}
Combined with  (\ref{z}), (\ref{neg}), and the triangle inequality, this yields
\begin{equation} \label{mass} \int_0^1 |\widehat{f_A}(\alpha)|^2|S_M(\alpha)| d\alpha \geq \delta^2NM^2/4.
\end{equation} 
By adapting traditional exponential sum estimates to this two-variable setting, and at one point carefully exploiting that $\Delta(h)\neq 0$, we have that if $\delta \geq N^{-1/20}$, then
\begin{equation} \label{Smaj} |S_M(\alpha)| \ll_h M^2/q \quad \text{for} \ \alpha \in \mathbf{M}_q, \ q\leq\eta^{-1},
\end{equation}
and 
\begin{equation} \label{Smin} |S_M(\alpha)| \leq C\eta M^2 \leq \delta M^2/8 \quad \text{for} \ \alpha \in \mathfrak{m}, 
\end{equation} provided we choose $c_0\leq1/8C.$ We prove and discuss these estimates in detail in Section \ref{expest}.

\noindent By (\ref{Smin}) and Plancherel's Identity (\ref{planch}), we have 
\begin{equation*}\int_{\mathfrak{m}}  |\widehat{f_A}(\alpha)|^2|S_M(\alpha)| d\alpha \leq \delta^2NM^2/8,
\end{equation*}
which by (\ref{mass}) yields
\begin{equation} \label{Mmass} \int_{\mathfrak{M}}  |\widehat{f_A}(\alpha)|^2|S_M(\alpha)| d\alpha \geq \delta^2NM^2/8.
\end{equation} By (\ref{Smaj}) and (\ref{Mmass}) we have 
\begin{equation} \label{almost} \delta^2 NM^2 \ll_h \sum_{q=1}^{\eta^{-1}} \frac{M^2}{q} \int_{\mathbf{M}_q} |\widehat{f_A}(\alpha)|^2 d \alpha. 
\end{equation} 
We then make use of the following proposition, a more general version of which can be found in Proposition 5.6 of \cite{Rice}, which exploits the more inclusive definition of $\mathbf{M}'_q$ as compared with $\mathbf{M}_q$.

\begin{proposition}\label{rstrick} If $\eta^2 N >2Q^2$,  then $$\max_{q\leq Q} \int_{\mathbf{M}'_q}|\widehat{f_A}(\alpha)|^2 d\alpha \geq \frac{1}{2} \sum_{q=1}^Q q^{-1}\int_{\mathbf{M}_q}|\widehat{f_A}(\alpha)|^2 {d}\alpha. $$ 
\end{proposition} 

\begin{proof} By (\ref{majdisj}) we have \begin{align*} Q \max_{q\leq Q} \int_{\mathbf{M}'_q}|\widehat{f_A}(\alpha)|^2 {d}\alpha &\geq \sum_{q=1}^Q  \int_{\mathbf{M}'_q}|\widehat{f_A}(\alpha)|^2 {d}\alpha \\ &= \sum_{q=1}^Q  \sum_{r|q}  \int_{\mathbf{M}_r}|\widehat{f_A}(\alpha)|^2 {d}\alpha \\ &= \sum_{r=1}^Q \lfloor Q/r\rfloor \int_{\mathbf{M}_r}|\widehat{f_A}(\alpha)|^2 {d}\alpha  \\ &\geq \frac{Q}{2}\sum_{r=1}^Q r^{-1}\int_{\mathbf{M}_r}|\widehat{f_A}(\alpha)|^2 {d}\alpha, 
\end{align*}  and the proposition follows.
\end{proof} \noindent Lemma \ref{sqL2} then follows immediately from (\ref{almost}) and Proposition \ref{rstrick}. \qed

\section{Exponential Sum Estimates} \label{expest}
In this section, we carefully adapt traditional exponential sum estimates in order to establish the crucial upper bounds (\ref{Smaj}) and (\ref{Smin}). For the entirety of the section, we fix a nonzero binary quadratic form $$h(x,y)=b_1x^2+b_2xy+b_3y^2 \in \Z[x,y].$$ Unlike in previous sections, we do not make the perpetual assumption that $\Delta(h)=b_2^2-4b_1b_3\neq 0$, but rather enforce this condition only when necessary.

\subsection{Major arc estimates} We begin our pursuit of (\ref{Smaj}) by establishing an asymptotic formula for the relevant exponential sum near rationals with small denominator. To achieve this goal, we make multiple appeals to the following standard formula, which is simply integration by parts applied to an appropriate Riemann-Stieltjes integral.

\begin{lemma}[Abel's Partial Summation Formula] \label{psf} If $\phi: \R \to \C$ is continuously differentiable, $f:\N \to \C$, $F(x)=\sum_{1\leq n \leq x}f(n)$, and $M>0$, then $$\sum_{1\leq n \leq M}f(n)\phi(n)= F(M)\phi(M)-\int_0^M F(x)\phi'(x) dx.  $$
\end{lemma}

\noindent We now proceed with the asymptotic formula, obtained by applying Lemma \ref{psf} one variable at a time.

\begin{lemma} \label{majform} If $a,q \in \N$, $\alpha=a/q+\beta$, and $M>0$, then $$S_M(\alpha)= \sum_{1\leq m,n \leq M} e^{2\pi i h(m,n)\alpha} = q^{-2}G(a,q)\int_{0}^M\int_{0}^M e^{2\pi ih(x,y)\beta} dxdy +O(qM(1+JM^2\beta)), $$ where $J=|b_1|+|b_2|+|b_3|$ and $$G(a,q)=\sum_{r,s=0}^{q-1} e^{2\pi i h(r,s)a/q}.$$

\end{lemma}

\begin{proof} For each fixed $1\leq m \leq M$ and $y>0$, we see that \begin{equation}\label{smy} S^m_y(a/q)=\sum_{1\leq n \leq y} e^{2\pi i h(m,n)a/q}=\sum_{s=0}^{q-1}e^{2\pi i h(m,s)a/q} \left| \{1\leq n \leq y \ : \ n\equiv s \text{ mod } q\}\right|=\frac{y}{q}G_m(a,q)+O(q), \end{equation} where $$ G_m(a,q)=\sum_{s=0}^{q-1}e^{2\pi i h(m,s)a/q}.$$
Then, letting $h_y=\frac{\partial h}{\partial y}$ and successively applying Lemma \ref{psf}, (\ref{smy}), and integration by parts, we have \begin{align*}S^m_M(\alpha) &=\sum_{1\leq n \leq M} e^{2\pi i h(m,n)a/q}e^{2\pi i h(m,n)\beta}\\ &= S^m_M(a/q)e^{2\pi i h(m,M)\beta}-\int_0^M S^m_y(a/q) (2\pi i h_y(m,y)\beta) e^{2\pi i h(m,y)\beta}dy \\ &= q^{-1}G_m(a,q)\left(Me^{2\pi i h(m,M)\beta}-\int_0^M y2\pi ih_y(m,y)\beta e^{2\pi i h(m,y)\beta}dy \right)+O(q(1+JM^2\beta)) \\ &=q^{-1}G_m(a,q)\int_0^M e^{2\pi i h(m,y)\beta}dy +O(q(1+JM^2\beta)).\end{align*} Similarly, summing in $m$ we have \begin{equation*} \tilde{S}_x(a/q)=\sum_{1\leq m \leq x} G_m(a,q)= \sum_{r=0}^{q-1}G_r(a,q)\left| \{1\leq m \leq x \ : \ m\equiv r \text{ mod } q\}\right|=\frac{x}{q}G(a,q)+O(q), \end{equation*} and, letting $h_x=\frac{\partial h}{\partial x}$, we apply the same sequence of steps to see that $S_M(\alpha)$ equals \begin{align*}&q^{-1}\sum_{1\leq m \leq M}G_m(a,q) \int_0^M e^{2\pi i h(m,y)\beta}dy +O(qM(1+JM^2\beta)) \\  =  &q^{-1} \left(\tilde{S}_M(a/q)\int_0^M e^{2\pi i h(M,y)\beta}dy-\int_0^M\int_0^M \tilde{S}_x(a/q)(2\pi i h_x(x,y)\beta)e^{2\pi i h(x,y)\beta} dxdy \right)+O(qM(1+JM^2\beta)) \\ =  & q^{-2}G(a,q)\left(M\int_0^M e^{2\pi i h(M,y)\beta}dy-\int_0^M\int_0^M x(2\pi i h_x(x,y)\beta)e^{2\pi i h(x,y)\beta}dxdy \right) +O(qM(1+JM^2\beta)) \\ =  & q^{-2}G(a,q)\int_{0}^M\int_{0}^M e^{2\pi ih(x,y)\beta} dxdy +O(qM(1+JM^2\beta)),\end{align*} and the formula is established.
\end{proof} 

\noindent The crucial denominator $q$ in (\ref{Smaj}) comes from the following result, previously discussed in Section \ref{boundsec}, which is the one and only juncture at which we require $\Delta(h)\neq 0$. This key ingredient, as well as the standard proof we recreate for Lemma \ref{weylbin}, rely on a technique known as \textit{Weyl differencing}, in which we take the modulus squared of the exponential sum in order to reduce the quadratic dependence in each variable to a linear dependence.
\begin{lemma} \label{gauss2} If $\Delta(h)\neq 0$ and $a,q\in \N$ with $(a,q)=1$, then $$\left|\sum_{r,s=0}^{q-1} e^{2\pi i h(r,s)a/q}\right| \ll_{h} q. $$

\end{lemma} 

\begin{proof} Fixing $a,q \in \N$ with $(a,q)=1$, exploiting that $|z|^2=z\bar{z}$ for any $z\in \C$, and changing variables $r'=r+t$, $s'=s+u$, we see that  \begin{align*} \left|\sum_{r,s=0}^{q-1} e^{2\pi i h(r,s)a/q}\right|^2 &= \sum_{r,r',s,s'=0}^{q-1}e^{2\pi i (h(r',s')-h(r,s))a/q} \\ &=\sum_{r,s,t,u=0}^{q-1}e^{2\pi i (h(r+t,s+u)-h(r,s))a/q} \\ &= \sum_{r,s,t,u=0}^{q-1} e^{2\pi i (2b_1rt+b_1t^2+b_2ru+b_2st+b_2tu+2b_3su+b_3u^2)a/q} \\ &= \sum_{t,u=0}^{q-1}e^{2\pi i h(t,u)a/q} \left(\sum_{r=0}^{q-1} e^{2\pi i (2b_1t+b_2u)ra/q}\right) \left(\sum_{s=0}^{q-1} e^{2\pi i (b_2t+2b_3u)sa/q} \right) \\ &= \sum_{t,u=0}^{q-1} e^{2\pi i h(t,u)a/q} \begin{cases} q^2 &\text{if } 2b_1t+b_2u \equiv b_2t+2b_3u \equiv 0 \text{ mod }q \\ 0 & \text{else} \end{cases}, \end{align*} where the last equality follows from the orthogonality relation \begin{equation*}\sum_{r=0}^{q-1}e^{2\pi i rb/q}=\begin{cases} q &\text{if } q\mid b \\ 0 &\text{else}\end{cases}. \end{equation*} Looking at the two congruence conditions above, multiplying the first expression by $b_2$, and multiplying the second expression by $2b_1$,  we get the system $$2b_1b_2t+b_2^2u\equiv 2b_1b_2t+4b_1b_3u \equiv 0 \text{ mod }q.$$ By subtracting the two resulting expressions we see that $q$ must divide $\Delta(h)u$. Letting $d=\gcd(q,\Delta(h))$, we have that $u$ must be one of the $d$ multiples of $q/d$, which each yield at most $\gcd(q,2b_1b_2)$ choices for $t$. In particular, if $\Delta(h)\neq 0$, then the number of simultaneous solutions is $O_h(1)$, and the lemma follows.
\end{proof} 

\subsection{Proof of (\ref{Smaj})} Returning to the setting of the proof of Lemma \ref{sqL2}, if $\alpha \in \mathbf{M}_q$ with $$q\leq \eta^{-1} \ll_h \delta^{-1} \leq N^{1/20} \ll M^{1/10},$$ then $\alpha=a/q+\beta$ with $$|\beta| < \frac{1}{\eta^2N} \ll_h N^{-9/10} \ll M^{-9/5}$$ for some $a$ with $(a,q)=1$. In this case, Lemma \ref{majform} tells us that $$S_M(\alpha)=q^{-2}G(a,q)\int_{0}^M\int_{0}^M e^{2\pi ih(x,y)\beta} dxdy+O_h(M^{1.3}). $$ Applying Lemma \ref{gauss2} and trivially bounding the double integral by $M^2$, we have $$|S_M(\alpha)|\ll_h M^2/q,$$ as claimed in (\ref{Smaj}). \qed

\subsection{Minor arc estimates} We begin our pursuit of (\ref{Smin}) with the following standard oscillatory integral estimate, which will allow us to exhibit (\ref{Smin}) in the case that $\alpha$ is fairly close to a rational with small denominator, but not so close as to lie in the major arcs. 

\begin{lemma}[Van der Corput's Lemma for Quadratic Polynomials] \label{vdc} If $g(x)=x^2+bx+c\in \R[x]$ and $I \subseteq \R$ is an interval, then $$\left|\int_I e^{2\pi i g(x)\beta} dx \right| \ll |\beta|^{-1/2}.  $$

\end{lemma}

\begin{proof} Fix $g(x)=x^2+bx+c \in \R[x]$ and an interval $I\subseteq \R$, and let $E=(I+b/2)\cap\{x: |x|\geq |\beta|^{-1/2}\}$, where $I+b/2$ denotes the translation of the interval $I$ by $b/2$.  We know that the measure of $(I+b/2)\setminus E$ is at most $2|\beta|^{-1/2}$, so we complete the square and change variables to see that \begin{align*}\left|\int_I e^{2\pi i g(x)\beta} dx\right| &= \left|\int_I e^{2\pi i ((x+b/2)^2-b^2/4+c)\beta} dx\right| \\ &= \left|\int_I e^{2\pi i (x+b/2)^2\beta} dx\right| \\ &= \left|\int_{I+b/2} e^{2\pi i y^2\beta} dy\right| \\ & \ll |\beta|^{-1/2}+ \left|\int_{E} e^{2\pi i y^2\beta} dy\right|. \end{align*}  Writing $$e^{2\pi i y^2\beta}= \frac{1}{4\pi i y \beta} \frac{d}{dx}(e^{2\pi i y^2\beta}), $$ we have by integration by parts that \begin{align*}\int_{E} e^{2\pi i y^2\beta} dy =  \left[\frac{e^{2\pi i y^2\beta}}{4\pi i y \beta} \right] + \int_{E} \frac{e^{2\pi i y^2\beta}}{4\pi i y^2 \beta} dy,\end{align*} where the expression in brackets is appropriately evaluated at endpoints of $E$.  By construction, $|y|\geq |\beta|^{-1/2}$ at each endpoint of $E$, and hence $$ \left|\int_{E} e^{2\pi i y^2\beta} dy \right| \ll |\beta|^{-1/2} + |\beta|^{-1}\int_{|y|\geq|\beta|^{-1/2}} \frac{1}{y^2} dy \ll |\beta|^{-1/2}, $$ which establishes the desired estimate.
\end{proof}  
\noindent With regard to estimating the double integral in the conclusion of Lemma \ref{majform}, since we assumed $h$ was not identically zero, we can relabel or make a linear change of variables to reduce to the case where $b_1\neq 0$.  Then, by applying Lemma \ref{vdc} to the integral in $x$ for every fixed $y$, we immediately get the following estimate.
\begin{corollary} \label{vdc2} If $M>0$, then \begin{equation} \label{vdc2e} \left|\int_0^M\int_0^M e^{2\pi i h(x,y)\beta} dx dy\right| \ll_{h}  M|\beta|^{-1/2}. \end{equation} 

\end{corollary}

\noindent For our final ingredient, we turn to the following traditional estimate, which we  utilize to establish (\ref{Smin}) when $\alpha$ is close to a denominator that is neither too small nor too large.

\begin{lemma}[Weyl's Inequality for Quadratic Polynomials] \label{weylbin} Suppose $g(x)=bx^2+cx+d \in \R[x]$, $b\in \N$, $a,q\in \N$, $t\geq 1$, and $x>0$. If $(a,q)=1$ and $|\alpha-a/q|\leq tq^{-2}$, then $$\left|\sum_{1\leq n \leq x}e^{2\pi i g(n) \alpha}\right| \ll \left( bx\log q+tx+btx^2/q+q\log q \right)^{1/2}.$$

\end{lemma}

\begin{proof} Letting $S$ denote the exponential sum we wish to estimate, we see that \begin{equation}\label{weylf} |S|^2 = \sum_{1\leq n,n' \leq x} e^{2\pi i (h(n')-h(n))\alpha} = x+ 2\Re\left(\sum_{1\leq n<n' \leq x}  e^{2\pi i (h(n')-h(n))\alpha}\right), \end{equation} where the $x$ accounts for terms where $n=n'$, and $\Re$ denotes the real part. With a change of variables $n'=n+h$, we have \begin{align*}\sum_{1\leq n<n' \leq x}  e^{2\pi i (h(n')-h(n))\alpha} & = \sum_{1\leq n \leq x-1} \sum_{1\leq h \leq x-n} e^{2\pi i (h(n+h)-h(n))\alpha}  \\ & =\sum_{1\leq n \leq x-1} \sum_{1\leq h \leq x-n} e^{2\pi(2bnh+h^2+ch)\alpha} \\ &= \sum_{1\leq h \leq x-1} e^{2\pi i (h^2+ch)\alpha} \sum_{1\leq n \leq x-h} e^{2\pi i(2bhn)\alpha}. \end{align*} Applying the geometric series formula to the inner sum and the triangle inequality gives us \begin{equation}\label{weylnext} \left|\sum_{1\leq n<n' \leq x}  e^{2\pi i (h(n')-h(n))\alpha} \right| \ll \sum_{1\leq h \leq x} \min \left\{x, \norm{2bh\alpha}^{-1} \right\} \leq \sum_{1\leq h \leq 2bx} \min \left\{x, \norm{h\alpha}^{-1} \right\}, \end{equation} where $\norm{\cdot}$ denotes the distance to the nearest integer. 
 
\noindent Fixing $q\in \N$ and breaking the sum in $h$ into intervals of length $q$, we have \begin{equation}\label{break}\sum_{1\leq h \leq 2bx} \min \left\{x, \norm{h\alpha}^{-1} \right\} \leq \sum_{1 \leq j \leq 2bx/q} \sum_{s=0}^{q-1} \min\left\{ x, {\norm{(qj+s)\alpha}}^{-1} \right\}.\end{equation} If $a \in \N$ with $|\alpha-a/q|\leq tq^{-2}$, we can write $\alpha=a/q+O(t/q^2)$, and hence \begin{equation*} (qj+s)\alpha= qj\alpha+\frac{sa}{q}+O(t/q).\end{equation*} Further, if we let $k$ be the nearest integer to $q^2j\alpha$, then $qj\alpha = k/q +O(t/q)$ and hence \begin{equation*} (qj+s)\alpha= \frac{sa+k}{q}+O(t/q). \end{equation*} Combined with (\ref{break}), this yields \begin{equation} \label{last} \sum_{1\leq h \leq 2bx} \min \left\{x, \norm{h\alpha}^{-1} \right\} \leq \sum_{1 \leq j \leq 2bx/q} \sum_{s=0}^{q-1} \min\left\{ x, \norm{\frac{sa+k}{q}+O(t/q)}^{-1}\right\}. \end{equation} If $(a,q)=1$, then as $s$ runs over all congruence classes modulo $q$, so does $sa$. In particular, the $O(t/q)$ error term dominates for at most $O(t)$ terms, and we have \begin{align*} \sum_{1 \leq j \leq 2bx/q} \sum_{s=0}^{q-1} \min\left\{ x, \norm{\frac{sa+k}{q}+O(t/q)}^{-1}\right\} \ll \sum_{1 \leq j \leq 2bx/q} \left(tx + \sum_{s=1}^{q/2} \frac{q}{s}\right) \ll (2bx/q+1)(tx+q\log q), \end{align*} which combines with (\ref{weylf}), (\ref{weylnext}), and (\ref{last}) to yield the desired estimate. 
\end{proof}

\noindent In the same way we deduce Corollary \ref{vdc2} from Lemma \ref{vdc}, we reduce to the case of $b_1\neq 0$ and apply Lemma \ref{weylbin} to the sum in $m$ for every fixed $n$ to immediately get the following estimate.

\begin{corollary} \label{weylbin2} Suppose $a,q\in \N$, $\alpha \in [0,1]$, and $x>0$. If $(a,q)=1$ and $|\alpha-a/q|\leq q^{-2}$, then \begin{equation} \label{weylbin2e} \left|\sum_{1\leq m,n \leq x}e^{2\pi i h(m,n)\alpha}\right| \ll_h x\left(x\log q+x^2/q+q\log q\right)^{1/2}.\end{equation}

\end{corollary} 

\noindent \textbf{Remark.} We note that under the assumption $\Delta(h)\neq 0$, the estimates (\ref{vdc2e}) and (\ref{weylbin2e}) can be improved to $|\beta|^{-1}$ and $$\left(x^4/q^2+(x^3/q+x^2+qx)\log q \right)^{1/2},$$ respectively. For the former, since it is in a continuous setting, one can simply use that if $b^2-4ac\neq 0$, then $$ax^2+bxy+cy^2= u^2+v^2$$ after an invertible linear change of variables, and then apply Lemma \ref{vdc} separately in $u$ and $v$. The latter estimate can be established by mimicking the two-variable Weyl differencing process, and exploitation of nonzero discriminant, exhibited in the proof of Lemma \ref{gauss2}. However, for the purposes of proving Theorem \ref{main}, we only require this sort of ``optimal cancellation" on the major arcs, so for the sake of brevity, and for the sake of exposing the components used in previous applications of this method, we leave the details of these improvements as exercises for the reader.

\subsection{Proof of (\ref{Smin})} Returning to the setting of the proof of Lemma \ref{sqL2}, we consider $\alpha \in \mathfrak{m}$. By the pigeonhole principle, there exists $1\leq q \leq M^{7/4}$ and $(a,q)=1$ such that $$|\alpha-a/q| \leq \frac{1}{qM^{7/4}}\leq\frac{1}{q^2}. $$ Writing $\alpha=a/q+\beta$, if $q\leq M^{1/4}$, then we have from Lemma \ref{majform} that \begin{equation}\label{form23} S_M(\alpha)=q^{-2}G(a,q)\int_{0}^M\int_{0}^M e^{2\pi ih(x,y)\beta} dxdy+O_h(M^{3/2}).  \end{equation} If $q\leq \eta^{-1}$, then it must be the case that $|\beta|>(\eta^2 N)^{-1}$, since otherwise we would have $\alpha \in \mathfrak{M}$. In this case, recalling that $N\ll_h M^2$ and $\eta \gg_h \delta \geq N^{-1/20} \gg_h M^{-1/10}$, it follows from (\ref{form23}), Corollary \ref{vdc2}, and trivially bounding $G(a,q)$ by $q^2$ that $$|S_M(\alpha)| \ll M|\beta|^{-1/2}+O_h(M^{3/2}) \ll_h \eta M^2. $$ If $\eta^{-1}\leq q \leq M^{1/4}$, then by (\ref{form23}), Lemma \ref{gauss2}, and trivially bounding the double integral by $M^2$, we have $$S_M(\alpha) \ll_h M^2/q+O_h(M^{3/2}) \ll_h \eta M^2. $$ Finally, if $M^{1/4}\leq q \leq M^{7/4}$, then by Corollary \ref{weylbin2} we have $$|S_M(\alpha)| \ll_h M(M\log q+M^2/q+q\log q)^{1/2} \ll M^{15/8} \ll_h \eta M^2, $$ and (\ref{Smin}) is established in all cases. \qed

\noindent \textbf{Acknowledgements:} The author would like to thank Neil Lyall who co-authored the expository note \cite{LR}, in the context of squares and shifted primes, that served as a template for this paper.


\begin{thebibliography}{10}    

 


\bibitem{BPPS} 
{\sc A. Balog, J. Pelik\'an, J. Pintz, E. Szemer\'edi}, {\em Difference sets without $\kappa$-th powers}, Acta. Math. Hungar. 65 (2) (1994), 165-187.



\bibitem{Furst}
{\sc H. Furstenberg}, {\em Ergodic behavior of diagonal measures and a theorem of {S}zemer\'edi on arithmetic progressions},
  J. d'Analyse Math, 71 (1977), 204-256.
 
\bibitem{Green} 
{\sc B. Green}, {\em On arithmetic structures in dense sets of integers}, Duke Math. Jour. 114 (2002) no.2, 215-238.

\bibitem{taoblog}
{\sc B. Green, T. Tao, T. Ziegler}, {\em A Fourier-free proof of the Furstenberg-S\'ark\"ozy theorem}, https://terrytao.wordpress.com/2013/02/28/a-fourier-free-proof-of-the-furstenberg-sarkozy-theorem/.


 
\bibitem{HLR}  
{\sc M. Hamel, N. Lyall, A. Rice}, {\em Improved bounds on S\'ark\"ozy's theorem for quadratic polynomials}, Int. Math. Res. Not. no. 8 (2013), 1761-1782 


\bibitem{KMF}
{\sc T. Kamae, M. Mend\`es France}, {\em van der Corput's difference theorem}, Israel J. Math. 31, no. 3-4, (1978), pp. 335-342.
\bibitem{lipan} 
{\sc H.-Z. Li, H. Pan}, {\em Difference sets and polynomials of prime variables}, Acta. Arith. 138, no. 1 (2009), 25-52.

\bibitem{Lucier2}
{\sc J. Lucier}, {\em Difference sets and shifted primes}, Acta. Math. Hungar. 120 (2008), 79-102.
   
\bibitem{Lucier}
{\sc J. Lucier}, {\em Intersective sets given by a polynomial}, Acta Arith. 123 (2006), 57-95.

\bibitem{Lyall}
{\sc N. Lyall}, {\em A new proof of S\'ark\"ozy's theorem}, Proc. Amer. Math. Soc. 141 (2013), 2253-2264.

\bibitem{LM} 
{\sc N. Lyall, \`A. Magyar}, {\em Polynomial configurations in difference sets}, J. Number Theory 129 (2009), 439-450.

\bibitem{LR}
{\sc N. Lyall, A. Rice}, {\em Two theorems of S\'ark\"ozy}, http://alexricemath.com/wp-content/uploads/2013/06/DoubleSarkozy.pdf




\bibitem{PSS}
{\sc J. Pintz, W. L. Steiger, E. Szemer\'edi}, {\em On sets of natural numbers whose difference set contains no squares}, J. London Math. Soc. 37 (1988),  219-231.

\bibitem{Roth}
{\sc K. F. Roth}, {\em On certain sets of integers}, J. London Math. Soc. 28 (1953), pp. 104-109.

\bibitem{Rice}
{\sc A. Rice}, {\em A maximal extension of the best-known bounds for the S\'ark\"ozy-Furstenberg Theorem}, Acta Arith. 187 (2019), 1-41.

\bibitem{thesis}
{\sc A. Rice}, {\em Improvements and extensions of two theorems of S\'ark\"ozy}, Ph.D. thesis, University of Georgia, 2012. http://alexricemath.com/wp-content/uploads/2013/06/AlexThesis.pdf. 


\bibitem{Rice2} 
{\sc A. Rice}, {\em S\'ark\"ozy's theorem for $\P$-intersective polynomials}, Acta Arith. 157 (2013), no. 1, 69-89.

\bibitem{Ruz2}
{\sc I. Ruzsa}, {\em Difference sets without squares}, Period. Math. Hungar. 15 (1984), 205-209.


\bibitem{Ruz}
{\sc I. Ruzsa, T. Sanders}, {\em Difference sets and the primes}, Acta. Arith. 131, no. 3 (2008), 281-301.


\bibitem{Sark1}
{\sc A. S\'ark\"ozy}, {\em On difference sets of sequences of integers I}, Acta. Math. Hungar. 31(1-2) (1978), 125-149.

\bibitem{Sark3}
{\sc A. S\'ark\"ozy}, {\em On difference sets of sequences of integers III}, Acta. Math. Hungar. 31(3-4) (1978), 355-386.




\bibitem{Slip} {\sc S. Slijep\v{c}evi\'c}, {\em A polynomial S\'ark\"ozy-Furstenberg theorem with upper bounds}, Acta Math. Hungar. 98 (2003),  275-280.

 


\end{thebibliography}
\end{document}